\documentclass[12pt]{article}
\usepackage{amssymb}
\usepackage{amsmath,amsthm}
\usepackage[latin1]{inputenc}
\usepackage{graphicx}
\usepackage{hyperref}
\usepackage{enumerate}
\usepackage{tikz}
\usepackage{tkz-graph}
\usepackage{mathrsfs}
\usepackage{verbatim}

\hypersetup{colorlinks=true, linkcolor=blue, citecolor=blue, urlcolor=blue}

\newtheorem{remark}{Remark}

\newtheorem{lemma}[remark]{Lemma}
\newtheorem{theorem}[remark]{Theorem}
\newtheorem{proposition}[remark]{Proposition}
\newtheorem{corollary}[remark]{Corollary}

\newtheorem{conjecture}[remark]{Conjecture}
\title{On the super domination number of graphs}

\author{Douglas J. Klein$^{1}$, Juan A. Rodr{\'\i}guez-Vel\'azquez$^{1,2}$,
 Eunjeong Yi$^{1}$
\\
\\
$^1${\small Texas A$\&$M University at Galveston}
\\{\small  Foundational Sciences,  Galveston, TX 77553, United States}
\\
$^2${\small 
Universitat Rovira i Virgili, Departament d'Enginyeria
Inform\`{a}tica i Matem\`{a}tiques}\\
{\small
Av. Pa\"{i}sos Catalans 26, 43007
Tarragona, Spain}
\\{\small Email addresses: kleind\@@tamug.edu, juanalberto.rodriguez\@@urv.cat,yie\@@tamug.edu}
}

\begin{document}
\maketitle

\begin{abstract}
The  open neighbourhood of a vertex $v$ of a graph $G$ is the set $N(v)$ consisting of all vertices 
adjacent to $v$ in $G$.  For $D\subseteq V(G)$, we define $\overline{D}=V(G)\setminus D$. A~set  $D\subseteq V(G)$ is  called  a  super  dominating  set  of $G$      if for every vertex  $u\in \overline{D}$,  there exists $v\in D$ such that  $N(v)\cap \overline{D}=\{u\}$. The   super domination number  of $G$ is the minimum cardinality among all super dominating sets in $G$.
In this article, we obtain  closed formulas and  tight bounds for the super domination number of $G$ in terms of several invariants of $G$. Furthermore, the particular cases of corona product graphs and Cartesian product graphs are considered.
\end{abstract}

{\it Keywords: Super domination number; Domination number; Cartesian product; Corona product.}  

{\it AMS Subject Classification numbers: 05C69; 05C70  ; 05C76}

\section{Introduction}

The {\it open neighbourhood} of a vertex $v$ of a graph $G$ is the set $N(v)$ consisting of all vertices 
adjacent to $v$ in $G$.  For $D\subseteq V(G)$, we define $\overline{D}=V(G)\setminus D$. A set $D\subseteq V(G)$  is {\it dominating} in $G$ if every vertex in $\overline{D}$ 
has at
least one neighbour in $D$, \textit{i.e}., $N(u)\cap D\ne \emptyset$ for every $u\in \overline{D}$. The {\it domination number} of $G$, denoted by  
$\gamma (G)$, is the minimum cardinality among all dominating sets in $G$. A dominating set of cardinality $\gamma(G)$ is called a $\gamma(G)$-set. The reader is referred to the books \cite{Haynes1998a,Haynes1998}
 for details on domination in graphs.

 A  set  $D\subseteq V(G)$ is  called  a  \textit{super  dominating  set} of $G$   if for every vertex  $u\in \overline{D}$,  there exists $v\in D$ such that  
\begin{equation}\label{DefinitionPrivateNeighbour}
N(v)\cap \overline{D}=\{u\}.
\end{equation}
If $u$ and $v$ satisfy \eqref{DefinitionPrivateNeighbour}, then we say that $v$ is a \textit{private neighbour of $u$ with respect to} $\overline{D}$. The {\it super domination number} of $G$,  denoted by $\gamma_{\rm sp} (G)$, is the minimum cardinality among all super dominating sets in $G$. A super dominating set of cardinality $\gamma_{\rm sp}(G)$ is called a $\gamma_{\rm sp}(G)$-set. The study of super domination in graphs was introduced in \cite{MR3396565}. We recall some results on the 
  extremal values of $\gamma_{\rm sp} (G)$.

\begin{theorem}{\rm \cite{MR3396565}} 
\label{theorem1} 
Let $G$ be a graph of order $n$. The following assertions hold.
\begin{itemize}
\item $\gamma_{\rm sp}(G)=n$ if and only if $G$ is an empty graph.
\item $\gamma_{\rm sp}(G)\ge  \lceil \frac{n}{2}\rceil$.
\item $\gamma_{\rm sp}(G)=1$ if and only if $G\cong K_1$  or $G\cong K_2$.
\end{itemize}
\end{theorem}

It is well known that for any graph $G$ without isolated vertices, $1\leq \gamma(G)\le  \lceil \frac{n}{2}\rceil $. As noticed in \cite{MR3396565}, from the theorem above  we have that for any  graph $G$ without isolated vertices,
\begin{equation}\label{TrivialBoundsontheSuperDominationNumber}
1\leq \gamma(G)\le  \left \lceil \frac{n}{2} \right\rceil \le \gamma_{\rm sp}(G)\leq n-1.
\end{equation}

Connected graphs with $\gamma_{\rm sp}(G)=\frac{n}{2}$ were characterized in \cite{MR3396565}, while all graphs with $\gamma_{\rm sp}(G)=n-1$ were characterized in 
\cite{Dettlaff-LemanskaRodrZuazua2017}.

The following examples were previously shown in \cite{MR3396565}.
\begin{enumerate}[{\rm(a)}]
\item  For a complete graph $K_n$ with $n \ge  2$, $\gamma_{\rm sp}(K_n)=n-1$.
\item For a star graph $K_{1,n-1},$ $\gamma_{\rm sp}(K_{1,n-1})=n-1$.
\item For a complete bipartite graph $K_{r,t}$ with $\min\{r,t\}\ge 2$, $\gamma_{\rm sp}(K_{r,t})=r+t-2$.
\end{enumerate}
The three cases above can be generalized as follows. Let $K_{a_1,a_2,\ldots, a_k}$ be the complete $k$-partite graph of order $n=\sum_{i=1}^{k}a_i$.  If at most one value $a_i$ is greater than one, then $\gamma_{\rm sp}(K_{a_1,a_2,\ldots, a_k})=n-1$ as in such a case $ K_{a_1,a_2,\ldots, a_k}\cong K_n$ or $K_{a_1,a_2,\ldots, a_k}\cong K_{n-a_i}+N_{a_i}$, where $G+H$ denotes the join of graphs $G$ and $H$. As shown in \cite{Dettlaff-LemanskaRodrZuazua2017}, these cases are included in the family of graphs with super domination number equal to $n-1$. On the other hand, it is not difficult to show that if there are at least two values $a_i,a_j\ge 2$, then  $\gamma_{\rm sp}(K_{a_1,a_2,\ldots, a_k})=n-2$. We leave the details to the reader. In summary, we can state the following.
\begin{equation*}
\gamma_{\rm sp}(K_{a_1,a_2,\ldots, a_k})=\left\{
\begin{array}{ll}
n-1 & \mbox{ if  at most one value } \ a_i \text{ is greater than one.}\\
n-2 & \mbox{ otherwise. }
\end{array}\right.
\end{equation*}

The particular case of paths and cycles was studied in \cite{MR3396565}.

\begin{theorem}{\rm \cite{MR3396565}}\label{Formula cycles and paths}
For any integer  $n\ge 3$,
$$\gamma_{\rm sp}(P_n)=\left\lceil\frac{n}{2} \right \rceil.$$

\begin{eqnarray*}\gamma_{\rm sp}(C_n)= \left \{ \begin{array}{ll}

\left \lceil\frac{n}{2} \right \rceil, & n\equiv 0,3\pmod 4 ;
\\
\\
\left \lceil\frac{n+1}{2} \right \rceil, & \text{otherwise}.
\end{array} \right
.\end{eqnarray*}
\end{theorem}

%%%%%%%%%%

It was shown in \cite{Dettlaff-LemanskaRodrZuazua2017}  that the problem of computing $\gamma_{\rm sp} (G)$ is NP-hard. This suggests that finding the super domination number for special classes of graphs or obtaining good bounds on this invariant is worthy of investigation. In particular, the super domination number of lexicographic product graphs   and joint graphs was studied in \cite{Dettlaff-LemanskaRodrZuazua2017} and the case of rooted product graphs was studied in \cite{MR3439855}.
 In this article we study the problem of finding exact values or sharp bounds for the super domination number of graphs. 
 
 The article is organised as follows. In Section \ref{SectionOtherParameters}, we study the relationships between $\gamma_{\rm sp}(G)$ and several parameters of $G$, including the number of twin equivalence classes, the domination number, the secure domination number, the matching number, the $2$-packing number, the vertex cover number, etc. In Section \ref{SectionCorona}, we obtain a closed formula for the super domination number of any corona graph, while in Section \ref{SectionCartesian} we study the problem of finding the exact values or sharp bounds for the super domination number of
Cartesian product graphs and express these in terms of invariants of the factor graphs. 

\section{Relationship between the super domination number and  other parameters of graphs}\label{SectionOtherParameters}
 %%%% Matching number

A \emph{matching}, also called an independent edge set, on a graph $G$ is a set of edges of $G$ such that no two edges share a vertex in common. A largest matching (commonly known as a maximum matching or maximum independent edge set) exists for every graph. The size of this maximum matching is called the  \emph{matching number} and is denoted  by $\alpha'(G)$.

\begin{theorem}\label{BoundMatchingNumber}
For any graph $G$ of order $n$,
$$ \gamma_{\rm sp}(G)\ge n-\alpha'(G).$$
\end{theorem}

\begin{proof}
Let $D$ be a $\gamma_{\rm sp}(G)$-set and let $D^*\subseteq D $ be a set  of cardinality $|D^*|=|\overline{D}|$ such that for every $u\in \overline{D}$ there exists $u^*\in D^*$ such that 
$N(u^*)\cap \overline{D}=\{u\}.$
 Since $$M=\{u^*u\in E(G):\, u^*\in D^* \text{ and }u\in \overline{D}  \}$$ is a matching, we have that 
$n-\gamma_{\rm sp}(G)=|\overline{D}|=|M|\le \alpha'(G),$ as required.
\end{proof}

As a simple example of a graph where the bound above is achieved we can take $G\cong K_1+(\cup^k K_2)$. In this case $\alpha'(G)=k$, $n=2k+1$ and $\gamma_{\rm sp}(G)=k+1=n-\alpha'(G).$ Another example is $\gamma_{\rm sp}(C_{n})=\lceil\frac{n}{2}\rceil=n -\alpha'(C_{n})$ whenever $n\equiv 0,3\pmod 4 .$

%%%%%%%%%% (K\"{o}nig 1931, Egerv\'{a}ry 1931)
 A {\it vertex cover} of  $G$ is a set $X\subset V(G)$  such that each 
edge of $G$ is incident to at least one vertex of $X$.  A minimum vertex cover 
is a vertex cover of smallest possible cardinality. The {\it vertex cover 
number} $\beta(G)$ is the cardinality of a minimum vertex cover of $G$. 
%A vertex cover of cardinality $\beta(G)$ is called a $\beta(G)$-set.

\begin{theorem}{\rm (K\"{o}nig \cite{Konig1931}, Egerv\'{a}ry \cite{Egervary1931})}\label{KonigEgervystheorem}
For bipartite graphs the size of a maximum matching equals the size of a 
minimum vertex cover.
\end{theorem}

We   use Theorems \ref{BoundMatchingNumber} and   \ref{KonigEgervystheorem} to derive the follo\-wing result. 

\begin{theorem}\label{LowerBoundSecureDomBipartiteCover}
For any bipartite graph $G$ of order $n$,
$$ \gamma_{\rm sp}(G)\ge n-\beta(G).$$
\end{theorem}

The bound above is attained, for instance, for any star graph and for any hypercube graph. It is well-known that for the hypercube $Q_k$, $\beta(Q_k)=2^{k-1}$ (see,
for instance, \cite{MR949280}) and in Section \ref{SectionCartesian} we will show that   $\gamma_{\rm sp}(Q_k)=2^{k-1}$.

An \emph{independent set} of $G$ is a set $X \subseteq V(G)$ such that no two vertices in $X$ are adjacent in $G$, and the \emph{independence number} of $G$, $\alpha(G)$, is the cardinality of a largest independent set of $G$.

The following well-known result, due to Gallai, states the relationship between the independence number and the vertex cover number of a graph.

\begin{theorem}\label{GallaiTheorem}
{\rm (Gallai's theorem)} For any graph $G$ of order $n$,
$$\alpha(G)+\beta(G)=n.$$
\end{theorem}

From Theorems \ref{LowerBoundSecureDomBipartiteCover} and \ref{GallaiTheorem} we deduce the following tight bound.

\begin{theorem}\label{LowerBoundSecureDomBipartiteIndep}
For any bipartite graph $G$ of order $n$,
$$ \gamma_{\rm sp}(G)\ge \alpha(G).$$
\end{theorem}

%%%%Secure domination
A set $S\subseteq V(G)$ is said to be a \textit{secure dominating set} of $G$ if  it is a dominating set and for every $v\in \overline{S}$ there exists $u\in N(v)\cap S$ such that  $(S\setminus \{u\})\cup \{v\}$ is a dominating set.  The \emph{secure domination number}, denoted by $\gamma_s(G)$,  is the minimum cardinality among all secure dominating sets. 
This type of domination was introduced  by Cockayne et al. in \cite{MR2137919}.
\begin{theorem}
For any graph $G$,
$$\gamma_{\rm sp}(G)\ge \gamma_s(G).$$
\end{theorem}
\begin{proof}
Let $S\subseteq V(G)$ be a $\gamma_{\rm sp}(G)$-set. For each $v\in \overline{S}$ let $v^*\in S$ such that $N(v^*)\cap \overline{S}=\{v\}$. Obviously, $S$ and $(S\setminus \{v^*\})\cup \{v\}$ are dominating sets. Therefore, $S$ is a secure dominating set and so $ \gamma_s(G)\le |S|=\gamma_{\rm sp}(G)$.
\end{proof}
The inequality above is tight. For instance, $ \gamma_s(K_{1,n-1})=\gamma_{\rm sp}(K_{1,n-1})=n-1.$ Moreover, the equality is also achieved whenever  $\gamma_{\rm sp}(G)=\gamma(G)$, since  $\gamma(G) \le \gamma_s(G)\le \gamma_{\rm sp}(G)$. This case will be discussed in Corollary \ref{CorDom_SuperdomN/2}.

%%%%%%%%%%%%% Twin vertices
The {\it  closed
neighbourhood} of a vertex  $v$ is defined to be  $N[v]=N(v)\cup \{v\}$.
We define the \textit{twin  equivalence relation} ${\cal R}$ on $V(G)$ as follows:
$$x {\cal R} y \longleftrightarrow  N[x]=N[y] \; \; \mbox{\rm or } \; N(x)=N(y).$$

We have three possibilities for each twin equivalence class $U$:

\begin{enumerate}[(a)]
\item $U$ is a singleton set, or
\item $|U|>1$ and  $N(x)=N(y)$ for any $x,y\in U$% (and case (a) does not apply)
, or
\item  $|U|>1$ and   $N[x]=N[y]$ for any $x,y\in U$. % (and case (a) does not apply).
\end{enumerate}

We will  refer to the types (a) (b) and (c) classes as the \textit{singleton, false and true twin  equivalence classes}, respectively.

Let us see three different examples of non-singleton equivalence classes. An example of a graph where every equivalence class is a true twin equivalence class is 
$K_r+(K_s\cup K_t)$, $r,s,t\ge 2$. In this case, there are three equivalence
classes composed of $r,s$ and $t$  true twin vertices, respectively. 
As an example where every class is composed of false twin vertices, we take the complete bipartite graph $K_{r,s}$, $r,s\ge 2$. 
Finally, the graph $K_r+N_s$,  $r,s\ge 2$, has two equivalence classes and one of them is composed of $r$ true twin vertices and the other one is composed of $s$ false twin vertices. 
On the other hand, $K_1+(K_r\cup N_s)$, $r,s\ge 2$, is an example where one class is singleton, one class is composed of true twin vertices and the other one is composed of false twin vertices.

%%%%% Begin NEW 01-03-2018

The following straightforward lemma will be very useful to prove our next theorem.

\begin{lemma}\label{LemmaTwins}
Let $G$ be a graph and $D$  a $\gamma_{\rm sp}(G)$-set. Let $D^*\subseteq D $ such that $|D^*|=|\overline{D}|$  and for every $u\in \overline{D}$ there exists $u^*\in D^*$ such that 
$N(u^*)\cap \overline{D}=\{u\}.$ If $U\subseteq V(G)$ is a twin equivalence class, then $|U\cap \overline{D}|\le 1$  and $|U\cap D^*|\le 1$.
\end{lemma}

\begin{theorem}\label{ThTwins}
For any  graph $G$  of order $n$ having $t$ twin equivalence classes,
$$\gamma_{\rm sp}(G)\ge n-t.$$
Furthermore, if $G$ is connected and $t\ge 3$, then $$\gamma_{\rm sp}(G)\ge n-t+1.$$
\end{theorem}

\begin{proof}
Let $\{B_1,B_2,\dots ,B_t\}$ be the set of twin equivalence classes of $G$ and let $D$ be a $\gamma_{\rm sp}(G)$-set. 
%If there are two different vertices $x,y\in B_i\setminus W$, then $W\cap N(x)=W\cap N(y)$, which is a contradiction. Thus, 
By Lemma \ref{LemmaTwins} for every twin equivalence class we have  $|B_i\cap D|\ge |B_i|-1$, which implies that $\gamma_{\rm sp}(G)=|D|\ge \sum_{i=1}^t|B_i|-t=n-t$.

Suppose that $\gamma_{\rm sp}(G)=n-t$. In such a case, by Lemma \ref{LemmaTwins} we have that $|B_i\cap \overline{D}|= 1$  and $|B_i\cap D^*|= 1$ for every twin equivalence class.  From now on we assume that $G$ is connected and $t\ge 3$. Hence, there exist three twin equivalence classes $B_i,B_j,B_l$  such that every vertex in $B_j$ is adjacent to every vertex in $B_i\cup B_l$, 
%$N(B_j)\supseteq (B_i\cup B_l)$,
 and six vertices $a,b\in B_i$, $x,y\in B_j$ and $u,v\in B_l$ such that $a,x,u\in \overline{D}$ and $b,y,v\in D^*$. Thus, $N(y)\cap \overline{D}\supseteq \{a,u\}$, which is a contradiction. Therefore, the result follows.
\end{proof}

The bound $\gamma_{\rm sp}(G)\ge n-t$ is achieved by complete nontrivial graphs $G\cong K_n$, complete bipartite graphs $G\cong K_{r,s}$, where $r,s\ge 2$, and by the disjoin union of these graphs. The bound $\gamma_{\rm sp}(G)\ge n-t+1$ is achieved by the graph $G$ shown in Figure \ref{FigureTwins}, where there are four false twin equivalence classes and a singleton equivalence class. In this case, white-coloured vertices form a $\gamma_{\rm sp}(G)$-set.

%%%% Figure Twins
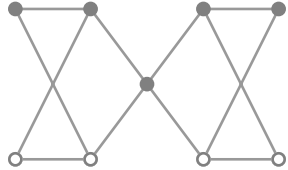
\begin{figure}[htb]
\begin{center}
\begin{tikzpicture}
[line width=1pt, scale=1]

\coordinate (V1) at (0.75,-1);
\coordinate (V2) at (1.75,-1);
\coordinate (V3) at (0.75,1);
\coordinate (V4) at (1.75,1);
\coordinate (V5) at (0,0);
\coordinate (V6) at (-0.75,1);
\coordinate (V7) at (-1.75,-1);
\coordinate (V8) at (-0.75,-1);
\coordinate (V9) at (-1.75,1);

\draw[black!40]  (V1)--(V2)--(V3)--(V4)--(V1)--(V5)--(V6)--(V9)--(V8)--(V7)--(V6);
\draw[black!40]  (V3)--(V5)--(V8);

\foreach \number in {1,...,9}{
\filldraw[gray]  (V\number) circle (0.08cm);
}

\foreach \number in {1,2,7,8}{
\filldraw[white]  (V\number) circle (0.08cm);
\draw[gray]  (V\number) circle (0.08cm);
}

%\foreach \number in {1,...,6}{
%\node [gray,above] at (V\number) {$\number $};
%\node [gray,below] at (V7) {$7$};
%}
\end{tikzpicture}
\end{center}
\vspace{-0,4cm}
\caption{A graph with $5$ twin equivalence classes where $\gamma_{\rm sp}(G)=n-4$.}
\label{FigureTwins} 
\end{figure}
%%%%%%%%%
%%%%%

The {\it open  neighbourhood} of a set $X\subseteq V(G)$ is defined to be  $N(X)=\cup_{x\in X}N(x)$, while the {\it closed neighbourhood} is defined to be $N[X]=X\cup N(X)$.
A set $S\subseteq V(G)$ is \textit{open irredundant} if for every $u\in S$, 
\begin{equation}\label{DefOpenIrredundant}
N(u)\setminus N[S\setminus\{u\}]\ne  \emptyset.
\end{equation}

\begin{theorem}{\rm \cite{Bollobas1979(c)}}\label{Teoerem:minimum dominating set which is open irredundant}
If a graph $G$ has no isolated vertices, then $G$ has a minimum dominating set which is open irredundant. 
\end{theorem}

\begin{theorem}\label{ThSuper-dom-order}
If a graph $G$ has no isolated vertices, then $$\gamma_{\rm sp}(G)\le n-\gamma(G).$$
\end{theorem}

\begin{proof}
If $G$ has no isolated vertices, by Theorem \ref{Teoerem:minimum dominating set which is open irredundant}, there exists an open irredundant set $S\subseteq V(G)$ such that $|S|=\gamma(G)$. Since every   $u\in S$ satisfies \eqref{DefOpenIrredundant}, we have that for every $u\in S$, there exits $v\in \overline{S}$ such that $N(v)\cap S=\{u\}$, which implies that $\overline{S}$ is a super dominating set. Therefore, $\gamma_{\rm sp}(G)\le |\overline{S}|= n-\gamma(G).$
\end{proof}

%%%%%% Example

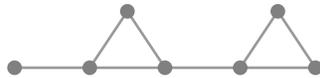
\begin{figure}[htb]
\begin{center}
\begin{tikzpicture}
[line width=1pt, scale=1]

\coordinate (V1) at (-2,0);
\coordinate (V2) at (-1,0);
\coordinate (V3) at (0,0);
\coordinate (V4) at (1,0);
\coordinate (V5) at (2,0);
\coordinate (V6) at (-0.5,0.75);
\coordinate (V7) at (1.5,0.75);

\draw[black!40]  (V1)--(V2)--(V3)--(V4)--(V5);
\draw[black!40]  (V2)--(V6)--(V3);
\draw[black!40]  (V4)--(V7)--(V5);

\foreach \number in {1,...,7}{
\filldraw[gray]  (V\number) circle (0.08cm);
}

%\foreach \number in {1,...,6}{
%\node [gray,above] at (V\number) {$\number $};
%\node [gray,below] at (V7) {$7$};
%}
\end{tikzpicture}
\end{center}
\vspace{-0,4cm}
\caption{For this graph $\gamma_{\rm sp}(G)=n-2=5$.}
\label{figMaxDegree} 
\end{figure}
%%%%%%%%%

The bound above is tight. For instance, it is achieved by the graph shown in Figure \ref{figMaxDegree}.
In Section  \ref{SectionCorona} we will show other examples of graphs where the bound above is achieved.

Since $\gamma_{sp}(G)\ge \frac{n}{2}$, we deduce the following consequence of  Theorem   \ref{ThSuper-dom-order}.

\begin{corollary}\label{CorDom_SuperdomN/2}
Let $G$ be a graph of order $n$.
If $\gamma(G)= \frac{n}{2}$, then  $\gamma_{sp}(G)= \frac{n}{2}$.
\end{corollary}

The converse of Corollary \ref{CorDom_SuperdomN/2} is not true. For instance, as we will see in Section \ref{SectionCartesian}, for the Cartesian product of $K_n$ times $K_2$ we have $\gamma_{sp}(K_n\Box K_2)=n$, while $\gamma(K_n\Box K_2)=2$.

A set $X\subseteq V(G)$ is  called a $2$-\textit{packing}
if $N[u]\cap N[v]=\emptyset $ for every pair of different vertices $u,v\in X$.
 The
$2$-\textit{packing number} $\rho(G)$ is the   cardinality
of any   largest $2$-packing of
$G$.  
It is well known that for any graph $G$, $\gamma(G)\ge \rho(G)$.
Meir and Moon \cite{MR0401519} showed in 1975 that $\gamma(T)= \rho(T)$ for any tree $T$.  We remark that in general, these $\gamma(T)$-sets and $\rho(T)$-sets  are not identical.

\begin{corollary}\label{CoroSup,packing}
Let  $G$  be a graph of order $n$. If $G$ does not have isolated vertices,
$$ \gamma_{\rm sp}(G)\le n-\rho(G).$$
\end{corollary}

 To show that the bound above is tight, we can take the graph shown in Figure \ref{figMaxDegree}. In Section  \ref{SectionCorona} we will show other examples.
 
 %%%%%%%%% Degree
 
As shown in \cite{Walikar1979}, the domination number of any graph of maximum degree $\Delta$ is bounded below by $\frac{n}{\Delta+1}$. Therefore, the following result is deduced from Theorem \ref{ThSuper-dom-order}.

\begin{corollary}\label{SecuredominationVS2-packing}
For any graph  $G$  of order $n$ and maximum degree $\Delta$, $$\gamma_{\rm sp}(G)\le \left\lfloor \frac{n\Delta}{\Delta +1} \right\rfloor.$$
\end{corollary}
The bound above is achieved, for instance, for any  graph with $\gamma_{\rm sp}(G)=n-1$, as in theses cases $\Delta=n-1$. An example of graph with $\Delta <n-1$ and $\gamma_{\rm sp}(G)=\left\lfloor \frac{n\Delta}{\Delta +1} \right\rfloor$ is the one shown in Figure \ref{figMaxDegree}.

By Theorems \ref{LowerBoundSecureDomBipartiteCover} and \ref{ThSuper-dom-order} we deduce the following result.

\begin{theorem}
Let $G$ be a bipartite graph. If $\gamma(G)=\beta(G)$, then 
$$\gamma_{\rm sp}(G)=\alpha(G).$$
\end{theorem}

Theorem \ref{SuperDominationStarTimeStar} will provide a family of Cartesian product graphs where $\gamma_{\rm sp}(G)=\alpha(G).$

The \emph{line graph} $L(G)$  of a simple non-empty graph $G$ is obtained by associating a vertex with each edge of $G$ and connecting two vertices of $L(G)$ with an edge if and only if the corresponding edges of $G$ have a vertex in common. 

\begin{theorem}\label{BoundPackingLine}
For any graph $G$ of order $n$, 
$$ \gamma_{\rm sp}(G)\le n-\rho(L(G)).$$
\end{theorem}

\begin{proof}
Let $M$ be a $2$-packing of $L(G)$ such that  $|M|=\rho(L(G))$. Let $X,X'\subset V(G)$ be two disjoint sets of cardinality $|M|$ such that $|e\cap  X|=1$ and $|e\cap  X'|=1$  for every $e\in M$, \emph{i.e.}, every edge in $M$ has an endpoint in $X$ and the other one  in $X'$. Since $M$ is  a $2$-packing of $L(G)$, for any $u,w\in X$ we have $N(v)\cap X'\cap N(w)=\emptyset$. Thus, $V(G)\setminus X$ is a super dominating set of $G$, as for every $u\in X$ there exists $v\in X'$ such that $N(v)\cap X=\{u\}$. Hence, 
$$\gamma_{\rm sp}(G)\le |V(G)\setminus X|= n-|X|=n-|M|=n-\rho(L(G)),$$
as required.
\end{proof}

To show that the bound above is tight we can take, for instance, both graphs shown in Figure \ref{fig:g9}. Notice that in these cases Theorem \ref{BoundPackingLine} gives a better result than Corollary \ref{CoroSup,packing}.

\section{Super domination in Corona product graphs}\label{SectionCorona}
The \emph{corona product graph} $G_1\odot G_2$ is defined as 
the graph obtained from $G_1$ and $G_2$ by taking one copy of $G_1$ and $|V(G_1)|$ copies of $G_2$ and joining by an edge each vertex from the $i^{th}$ copy of $G_2$ with the $i^{th}$ vertex of $G_1$ \cite{Frucht1970}. It is readily seen that $\gamma(G_1\odot G_2)=\rho(G_1\odot G_2)=|V(G_1)|$.  Therefore, The  bounds obtained in Theorem  \ref{ThSuper-dom-order} and Corollary \ref{CoroSup,packing} are tight, as for  the corona graph $G\cong G_1\odot K_r$ or $G\cong G_1\odot N_r$, where $G_1$ is an arbitrary graph of order $t$ and  $\gamma(G)=\rho(G)=t$.

\begin{theorem}
For any graph $G$ of order $n$ and any nonempty graph $H,$
$$\gamma_{\rm sp}(G\odot H)=n(\gamma_{\rm sp}(H)+1).$$
Furthermore, for any integer $r\ge 1$,
$$\gamma_{\rm sp}(G\odot N_r)=nr.$$
\end{theorem}

\begin{proof}We first assume that $H$ is nonempty. Let $r$ be the order of $H$, $U=\{u_1,\dots , u_n\}$  the vertex set of $G$, and $V_i$ the vertex set of the copy of $H$ associated to  $u_i$. Let $Y\subset V(H)$ be a $\gamma_{\rm sp}(H)$-set and $Y_i\subset V_i$ the set associated to $Y$ in the $i^{th}$ copy of $H$. 
It is readily seen that $U\cup \left(\cup_{i=1}^nY_i\right)$ is a super dominating set of $G\odot H$. Thus, $$\gamma_{\rm sp}(G\odot H)\le \displaystyle
\left|U\cup \left(\bigcup_{i=1}^nY_i\right)\right|=n+\sum_{i=1}^n|Y_i|=n(1+\gamma_{\rm sp}(H)).$$

Now, let $W$ be a $\gamma_{\rm sp}(G\odot H)$-set. If $u_i\not\in W$, for some $i$, then $V_i\subseteq W$, which implies that $W'=(W\setminus V_i)\cup Y_i\cup\{u_i\}$ is a super dominating set of $G\odot H$
and $|W'|\le |W|$, as $|Y_i|=\gamma_{\rm sp}(H)\le r-1$. Thus, from now on we can assume that $W$ is taken in such a way that $U\subset W$. Now, let $W_i=W\cap V_i$. If there exists $u_i\in U$ such that $|W_i|< |Y_i|$, then set of vertices of $H$ associated to $W_i$ is a super dominating set of $H$ and $|W_i|<\gamma_{\rm sp} (H)$, which is a contradiction. Hence, $|W_i|\ge |Y_i|=\gamma_{\rm sp}(H)$ for every $i$, which implies that
$$\gamma_{\rm sp}(G\odot H)=|W|=|U|+ \displaystyle
\sum_{i=1}^n|W_i|\ge n(1+ \gamma_{\rm sp}(H)).$$
Therefore, the first equality holds.

On the other hand, Theorems \ref{BoundMatchingNumber} and \ref{ThSuper-dom-order} imply that $\gamma_{\rm sp}(G\odot N_r)= nr$.
\end{proof}

An alternative proof for the result above can be derived from a formula obtained in \cite{MR3439855} for the super domination number of rooted product graphs. We leave the details to the reader.

\section{Super domination in Cartesian product graphs} \label{SectionCartesian}

The \textit{Cartesian product} of two graphs $G$ and $H$ is the graph $G\Box H$ whose vertex set is $V(G\Box H)=V(G)\times V(H)$ and two vertices $(g,h),(g',h')\in V(G\Box H)$ are adjacent in $G\Box H$ if and only if either

\begin{itemize}
\item $g=g'$ and $hh'\in E(H)$, or
\item $gg'\in E(G)$  and $h=h'$.
\end{itemize}

The Cartesian product is a straightforward and natural construction, and is in many respects the simplest graph product \cite{Hammack2011,Imrich2000}.  Hamming graphs, which includes hypercubes,  grid graphs and torus graphs are some particular cases of this product. The \emph{Hamming graph} $H_{k,n}$ is the Cartesian product of $k$ copies of the complete graph $K_n$. %\emph{i.e.},
%\[\begin{array}{c}
%            H_{k,n}=\underbrace{K_n\;\Box\; K_n\;\Box\; \cdots\; \Box\; K_n} \\
%            \;\;\;\;\;\;\;\;\;\;k \mbox{ times}
%          \end{array}
%\]
\emph{Hypercube} $Q_n$ is defined as $H_{n,2}$. Moreover, the \emph{grid graph} $P_k\Box P_n$ is the Cartesian product of the paths $P_k$ and $P_n$, the \emph{cylinder graph} $C_k\Box P_n$ is the Cartesian product of the cycle $C_k$ and the path $P_n$, and the \emph{torus graph} $C_k\Box C_n$ is the Cartesian product of the cycles $C_k$ and $C_n$.

This operation is commutative  in the sense that $G\Box H \cong H\Box G$, and is also associative in the sense  that  $(F \Box G)\Box H\cong F\Box (G\Box H)$. A Cartesian product  graph is connected if and only if both of its factors are connected.

This product has been extensively investigated from various perspectives. For instance, the most popular open problem in the area of domination theory is known as Vizing's conjecture. Vizing \cite{Vizing1968} suggested that the domination number of the Cartesian product of two graphs is at least as large as the product of domination numbers of its factors. Several researchers have worked on it, for instance, some partial results appears in \cite{Bresar2012,Hammack2011}.  For more information on structure and properties of the Cartesian product of graphs we refer the reader to \cite{Hammack2011,Imrich2000}.

Before obtaining our first result we need to introduce some additional notation. The set of all $\gamma_{\rm sp}(G)$-sets will be denoted by $\mathcal{S}(G)$.
For any $S\in \mathcal{S}(G)$   we define the set $\mathcal{P}(S)$ formed by subsets $S^*\subseteq S$ of cardinality $|S^*|=|\overline{S}|$ such that for every $u\in \overline{S}$ there exists $u^*\in S^*$ such that 
$N(u^*)\cap \overline{S}=\{u\}.$ With this notation in mind we define the following parameter which will be useful to study the super domination number of Cartesian product graphs.
$$
\lambda(G)=\max_{S\in \mathcal{S}(G),S^*\in\mathcal{P}(S)}\{|X|:\, X\subseteq S \text{ and } N(X) \cap (\overline{S}\cup S^*)=\emptyset\}.
$$

\begin{figure}[htb]
\begin{center}
\begin{tikzpicture}
[line width=1pt, scale=1]

\coordinate (V1) at (-2.5,1);
\coordinate (V2) at (-1.5,1);
\coordinate (V3) at (-0.5,1);
\coordinate (V4) at (0.5,1);
\coordinate (V5) at (1.5,1);
\coordinate (V6) at (2.5,1);
\coordinate (V7) at (0,0);

\draw[black!40]  (V1)--(V2);
\draw[black!40]  (V4)--(V3);

\foreach \number in {1,...,6}{
\draw[black!40]  (V\number)--(V7);  
}
\foreach \number in {1,...,7}{
\filldraw[gray]  (V\number) circle (0.08cm);
}

\foreach \number in {1,...,6}{
\node [gray,above] at (V\number) {$\number $};
\node [gray,below] at (V7) {$7$};
}
\end{tikzpicture}
\hspace{1cm}
\begin{tikzpicture}
[line width=1pt, scale=1]

\coordinate (V1) at (-1,1);
\coordinate (V2) at (0,1);
\coordinate (V3) at (1,1);
\coordinate (V4) at (0,0);

\draw[black!40]  (V1)--(V2);

\foreach \number in {1,...,3}{
\draw[black!40]  (V\number)--(V4);  
}
\foreach \number in {1,...,4}{
\filldraw[gray]  (V\number) circle (0.08cm);
}

\foreach \number in {1,...,3}{
\node [gray,above] at (V\number) {$\number $};
\node [gray,below] at (V4) {$4$};
}
\end{tikzpicture}
\end{center}
\caption{For the graph $G_1\cong K_1+(K_2\cup K_2\cup K_1\cup K_1)$ (on the left)  we have $\lambda(G_1)=2$, while  for the graph   $G_2\cong K_1+(K_2\cup  K_1)$ (on the right) we have   $\lambda(G_2)=1$.}
\label{fig:g9} 
\end{figure}
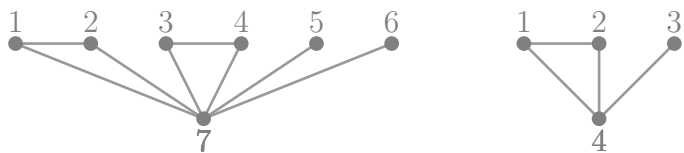

With the aim of clarifying what this notation means, we consider the graphs shown in Figure \ref{fig:g9}. For the graph $G_1\cong K_1+(K_2\cup K_2\cup K_1\cup K_1)$ (on the left)  we have $\gamma_{\rm sp}(G_1)=5$, $S_1=\{1,3,5,6,7\}\in \mathcal{S}(G_1)$  and $\mathcal{P}(S_1)=\{\{1,3\} \}$, while for the graph   $G_2\cong K_1+(K_2\cup  K_1)$ (on the right) we have that $\gamma_{\rm sp}(G_2)=3$, $S_2=\{1,3,4 \}\in \mathcal{S}(G_2)$  and $\mathcal{P}(S_2)=\{\{1\},\{4\}\}$. Notice that  $\lambda(G_1)=2$ and $\lambda(G_2)=1$.

If $G$ has $n$ vertices and $|N(v)|=n-1$, then $v$ is a \textit{universal vertex} of $G$. It is readily seen that the following remark holds.
\begin{remark}\label{RemarkUniversalVertex1}
Let $v$ be a universal vertex of a graph $G$ of order $n$ and let $S\in \mathcal{S}(G)$. If $v\in \overline{S}\cup S^*$ for some $S^*\in \mathcal{P}(S)$, then 
$\gamma_{\rm sp}(G)=n-1.$
\end{remark}

For instance, for the graph $G_2\cong K_1+(K_2\cup  K_1)$ shown in Figure  \ref{fig:g9} we have $\gamma_{\rm sp}(G_2)=n-1=3$,
$S=\{1,3,4\}\in \mathcal{S}(G_2)$ and  $S^*=\{4\}\in \mathcal{P}(S)$.

\begin{theorem}\label{UpperBopundGTimesH}
For any  graphs $G$  and $H$ of order $n\ge 2$ and  $n'\ge 2$, respectively,
$$ \left \lceil \frac{nn'}{2} \right\rceil \le \gamma_{\rm sp}(G\Box H)\le  n'\gamma_{\rm sp}(G)-\lambda(G)(n'-\gamma_{\rm sp}(H)).$$
\end{theorem}

\begin{proof}The lower bound is deduced from Theorem \ref{theorem1}, so we proceed to deduce the upper bound.
Let $S$ be a $\gamma_{\rm sp}(G)$-set, $S^*\in\mathcal{P}(S)$ and $X\subseteq S$ such that $|X|=\lambda(G)$ and $N(X) \cap (\overline{S}\cup S^*)=\emptyset$. We claim that for any $\gamma_{\rm sp}(H)$-set $S'$, the set 
$$W=V(G\Box H)\setminus \left(\left(\overline{S}\times V(H)\right) \cup \left( X\times \overline{S'}\right)\right)$$
 is a super dominating set of $G\Box H$. To see this we fix $(x,y)\in \overline{W}$. Notice that $x\in \overline{S}$ or $x\in X$, so that we differentiate these two cases. 
 
 Case 1: $x\in \overline{S}$. In this case, there exists $x^*\in S^*$ such that $N(x^*)\cap \overline{S}=\{x\}$. Since $\{x^*\}\times N(y)\subseteq W$, $(N(x^*)\times \{y\})\cap \overline{W}=(N(x^*)\cap \overline{S})\times \{y\}=\{(x,y)\}$ and $$N(x^*,y)=\left( \{x^*\}\times N(y)  \right)  \cup \left(N(x^*)\times \{y\}   \right),$$ 
we can conclude that $N(x^*,y) \cap \overline{W}=\{(x,y)\}.$ 

Case 2:  $x\in X$. In this case $N(x)\cap (\overline{S}\cup S^* )=\emptyset$ and $y\in \overline{S'}$. Since $S'$ is a super dominating set of $H$, there exists $y'\in S'$ such that $N(y')\cap \overline{S'}=\{y\}$. Also, if there exists $w\in N(x)\cap X$, then $S\setminus\{w\}$ is a super dominating set of $G$, which is a contradiction, so that $X$ is an independent set.
Hence, 
$$N(x,y')\cap \overline{W}=N(x,y')\cap (X\times \overline{S'})=\{x\}\times(N(y')\cap \overline{S'})=\{(x,y)\}.$$   
Therefore, $W$ is a super dominating set of $G\Box H$, which implies that $$\gamma_{\rm sp}(G\Box H)\le |W|=nn'-n'|\overline{S}|-|X\times \overline{S'}|=n'\gamma_{\rm sp}(G)-\lambda(G)(n'-\gamma_{\rm sp}(H)),$$ 
as required.
\end{proof}

As a direct consequence of Theorem \ref{UpperBopundGTimesH} we derive the following bound.
\begin{corollary}\label{InitialBound}
For any  graphs $G$  and $H$ of order $n\ge 2$ and  $n'\ge 2$, respectively,
$$\gamma_{\rm sp}(G\Box H)\le 
\min\{n'\gamma_{\rm sp}(G),n\gamma_{\rm sp}(H)\}.$$
\end{corollary}

We will see in Theorem \ref{UpperBopund2} that if $n\ge 3$, then $\gamma_{\rm sp}(K_n\Box K_3)=2n$, which implies that the bound given in Corollary \ref{InitialBound} is tight, as $\min\{3(n-1),2n\}=2n$ for $n\ge 3$.

The following result is a direct consequence of Theorem \ref{UpperBopundGTimesH} and Corolla\-ry~\ref{InitialBound}.

\begin{theorem}\label{CorollaryHalfTheOrder}
Let   $G$  and $H$ be two graphs of order $n\ge 2$ and  $n'\ge 2$, respectively. If $\gamma_{\rm sp}(G)=\frac{n}{2}$ or 
$\gamma_{\rm sp}(H)=\frac{n'}{2}$, then $$ \gamma_{\rm sp}(G\Box H)=\frac{nn'}{2}.$$
\end{theorem}

From the result above we have that for any graph $G$ of order $n\ge 2$, $$\gamma_{\rm sp}(G\Box K_2)=n.$$
Since the hypercube graph $Q_k$ is defined as $Q_k=Q_{k-1}\Box K_2$, for  $k\ge 2$, and $Q_1=K_2$, Corollary \ref{CorollaryHalfTheOrder} leads to the following result.  

\begin{remark}
For any integer $k\ge 1$,
$$\gamma_{\rm sp}(Q_k)=2^{k-1}.$$
\end{remark}

From Theorems \ref{theorem1} and \ref{Formula cycles and paths}, and Corollary \ref{InitialBound} we deduce the following result.

\begin{theorem}
Let $n\ge 3$ be an integer. The following statements hold for any graph  $H$ of order $n'\ge 2$.
\begin{itemize}
\item If $n\equiv 0\pmod 2$, then $\gamma_{\rm sp}(P_n\Box H)=\frac{nn'}{2}$.
\item If $n\equiv 1\pmod 2$, then $\frac{nn'}{2}\le \gamma_{\rm sp}(P_n\Box H)\le \frac{(n+1)n'}{2}$.
\item If $n\equiv 0\pmod 4$, then $\gamma_{\rm sp}(C_n\Box H)=\frac{nn'}{2}$.
\item If $n\equiv 1,3\pmod 4$, then $\frac{nn'}{2} \le \gamma_{\rm sp}(C_n\Box H)\le \frac{(n+1)n'}{2} $.
\item If $n\equiv 2\pmod 4$, then $\frac{nn'}{2} \le \gamma_{\rm sp}(C_n\Box H)\le \frac{(n+2)n'}{2} $.
\end{itemize}
\end{theorem}

As usual in domination theory, when studying a domination parameter, we can ask if a Vizing-like conjecture can be proved or formulated. 

\begin{conjecture}{\rm (Vizing-like conjecture)}
For any graphs $G$ and $H$,
$$\gamma_{\rm sp}(G\Box H)\ge \gamma_{\rm sp}(G )\gamma_{\rm sp}( H).$$
\end{conjecture}

 The above conjecture holds in the following case, which  is a direct consequence of Theorem \ref{CorollaryHalfTheOrder}.

\begin{remark}
Let   $G$  and $H$ be two graphs of order $n\ge 2$ and  $n'\ge 2$, respectively. If $\gamma_{\rm sp}(G)=\frac{n}{2}$ or $\gamma_{\rm sp}(H)=\frac{n'}{2}$, then $$\gamma_{\rm sp}(G\Box H)\ge \gamma_{\rm sp}(G )\gamma_{\rm sp}( H).$$
\end{remark}

In order to deduce another consequence of Theorem \ref{UpperBopundGTimesH}  we need to state the following lemma.
%%%%%
\begin{lemma}\label{RemarkUniversalVertex2}
Let $I(G)$ be the number of vertices of degree one of a graph $G$, and let $S\in \mathcal{S}(G)$. If there exists a universal vertex $v$ of $G$ such that $v\not\in \overline{S}\cup S^*$, for some $S^*\in \mathcal{P}(S)$, then $\lambda(G)\ge I(G)$.
\end{lemma}

\begin{proof}
Let $v$ be a universal vertex of $G$. If $I(G)=0$, then  we are done, so that we assume that $I(G)>0$. 

We first suppose that  $\gamma_{\rm sp}(G)=n-1$. If $G\cong K_{1,n-1}$, then for any $S\in \mathcal{S}(G)$ and  $S^*\in \mathcal{P}(S)$ the universal vertex of $G$ belongs to $ \overline{S}\cup S^*$. So we assume that  $G\not\cong K_{1,n-1}$. In such a case, for any pair of adjacent vertices $x,y\in V(G)\setminus \{v\}$ we have that $S=V(G)\setminus \{x\}\in \mathcal{S}(G)$ and $S^*=\{y\}\in \mathcal{P}(S)$, which implies that for any vertex $x$ of degree one, $N(x)\cap  (\overline{S}\cup S^*)=\emptyset$. Hence,  $\lambda(G)\ge I(G)$.

Now, suppose that $\gamma_{\rm sp}(G)\le n-2$. By Remark \ref{RemarkUniversalVertex1}, for any $S\in \mathcal{S}(G)$ and  $S^*\in \mathcal{P}(S)$ the universal vertex  $v$ does not belong to $ \overline{S}\cup S^*$, which implies that for any vertex $u$ of degree one, $N(u)\cap  (\overline{S}\cup S^*)=\emptyset$. Thus,  $\lambda(G)\ge I(G)$.
\end{proof}

By Lemma \ref{RemarkUniversalVertex2} we can derive a consequence of  Theorem \ref{UpperBopundGTimesH}  in which we replace the parameter $\lambda(G)$ by the number of vertices of degree one in $G$.

\begin{proposition}\label{CorollaryUpperboundUniversal}
Let $I(G)$ be the number of vertices of degree one of a graph $G$ of order $n$, and let $S\in \mathcal{S}(G)$. If there exists a universal vertex $v$ of $G$ such that $v\not\in \overline{S}\cup S^*$, for some $S^*\in \mathcal{P}(S)$, then for any graph $H$ of order $n'$,
$$\gamma_{\rm sp}(G\Box H)\le  n'\gamma_{\rm sp}(G)-I(G)(n'-\gamma_{\rm sp}(H)).$$
\end{proposition}

%%%%%%
In order to see that the bound above is tight, we can observe that for  $G\cong K_1+(K_2\cup  K_1)$ and $H\cong K_{n'}$, $n'\ge 3$, we have 
$\gamma_{\rm sp}(G\Box H)=3n'-1=  n'\gamma_{\rm sp}(G)-I(G)(n'-\gamma_{\rm sp}(H)).$
%%%%%%

Notice that, by Remark \ref{RemarkUniversalVertex1}, a particular case of 
Proposition \ref{CorollaryUpperboundUniversal} can be stated as follows.

\begin{corollary}
Let $I(G)$ be the number of vertices of degree one of a graph $G$ of order $n$ and maximum degree $n-1$. If $\gamma_{\rm sp}(G)\le n-2$,  then for any graph $H$ of order $n'$,
$$\gamma_{\rm sp}(G\Box H)\le  n'\gamma_{\rm sp}(G)-I(G)(n'-\gamma_{\rm sp}(H)).$$
\end{corollary}
%%%%%%

We now provide a tight bound on $\gamma_{\rm sp}(G\Box H)$ in terms of the order of $G$ and $H$.

\begin{theorem}\label{UpperBopund2}
For any nonempty graphs $G$  and $H$ of order $n\ge 2$ and  $n'\ge 2$, respectively,
$$\gamma_{\rm sp}(G\Box H)\le nn'-n-n'+4.$$
 Furthermore,
for any integers $n\ge 4$ and  $n'\ge 4$,$$\gamma_{\rm sp}(K_{n}\Box K_{n'})= nn'-n-n'+4$$
 and for any integer $n\ge 3$,$$\gamma_{\rm sp}(K_{n}\Box K_{3})= 2n.$$
\end{theorem}

\begin{proof}
Let $x_1,x_2\in V(G)$ and $y_1,y_2\in V(H)$ such that $x_2\in N(x_1)$ and $y_2\in N(y_1)$. Now, let $X\subseteq V(G\Box H)$ such that
$$\overline{X}=\left((V(G)\setminus \{x_1,x_2\})\times \{y_1\} \right)\cup \left( \{x_1\}\times (V(H)\setminus \{y_1,y_2\}) \right).$$
To check  that $X$ is a super dominating set of $G\Box H$ we only need to observe that for any $(x,y_1)\in \overline{X}$ there exists $(x,y_2)\in X$ such that $N(x,y_2)\cap \overline{X}=\{x,y_1\}$ and for any $(x_1,y)\in \overline{X}$ there exists $(x_2,y)\in X$ such that $N(x_2,y)\cap \overline{X}=\{x_1,y\}$. 
Hence, $\gamma_{\rm sp}(G\Box H)\le|X|= nn'-n-n'+4$.

To conclude the proof, it remains to consider the Cartesian product of complete graphs.  
Let $W$ be a $\gamma_{\rm sp}(K_{n}\Box K_{n'})$-set. 
Notice that if $(x,y)\in \overline{W}$, $(a,b)\in W$ and $N(a,b)\cap \overline{W}=\{(x,y)\}$, then  $x=a$ leads to $\overline{W}\cap (V(K_n)\times \{b\})=\emptyset$ 
 and $y=b$ leads to $\overline{W}\cap (\{a\}\times V(K_{n'}) )=\emptyset$. 
Furthermore,  if $(x,y),(x',y)\in \overline{W}$, then $\overline{W}\cap (\{x,x'\}\times V(K_{n'}))=\{(x,y),(x',y)\}$, as for any $y'\in V(K_{n'})\setminus \{y\}$ we have that $N(x',y)\subseteq N[(x,y)]\cup N[(x',y')]$ and $N(x,y)\subseteq N[(x',y)]\cup N[(x,y')]$. Analogously, if $(x,y),(x,y')\in \overline{W}$, then $\overline{W}\cap ( V(K_{n})\times \{y,y'\})=\{(x,y),(x,y')\}$.   Hence,  
\begin{equation}\label{LowerBoundComplementComplete}
\overline{W}|\le \max\{n,n',n+n'-4\}.
\end{equation}
Thus, if $n\ge 4$ and $n'\ge 4$, then $\gamma_{\rm sp}(K_{n}\Box K_{n'})\ge nn'-n-n'+4$, which implies that $\gamma_{\rm sp}(K_{n}\Box K_{n'}) = nn'-n-n'+4$.

Moreover, if $n\ge 3$, then Equation \eqref{LowerBoundComplementComplete} leads to $\gamma_{\rm sp}(K_{n}\Box K_{3})\ge 2n$. To conclude that $\gamma_{\rm sp}(K_{n}\Box K_{3})= 2n$ we only need to observe that for any $y,y'\in V(K_3)$ the set $V(K_n)\times \{y,y'\}$  is a super dominating set of $K_n\Box K_3$. Therefore, the result follows.
\end{proof}

The independence number of any Cartesian product graph is bounded below as follows. 

\begin{theorem}{\rm \cite{Vizing1963}}\label{VizingTheoremIndependence}
For any graphs $G$ and $H$ of order $n$ and $n'$, respectively,
$$\alpha(G\Box H)\ge \alpha(G)\alpha(H)+\min\{n-\alpha(G),n'-\alpha(H)\}.$$
\end{theorem}

From Theorems \ref{GallaiTheorem}, \ref{LowerBoundSecureDomBipartiteIndep} and \ref{VizingTheoremIndependence}   we deduce the following result.

\begin{theorem}\label{LoweboundSecureDomCartesianBipartite}
For any pair of bipartite graphs $G$ and $H$,
$$\gamma_{\rm sp}(G\Box H)\ge \alpha(G)\alpha(H)+\min\{\beta(G),\beta(H)\}.$$
\end{theorem}

\begin{theorem}\label{SuperDominationStarTimeStar}
For any integers $r\ge r'\ge 1$,
$$ \gamma_{\rm sp}(K_{1,r}\Box K_{1,r'})= rr'+1.$$
\end{theorem}

\begin{proof}By Theorem \ref{LoweboundSecureDomCartesianBipartite}, 
$\gamma_{\rm sp}(K_{1,r}\Box K_{1,r'}) \ge rr'+1$. 
Next, we proceed to show that $\gamma_{\rm sp}(K_{1,r}\Box K_{1,r'}) \le rr'+1$. Let $u_0$, $v_0$, and $U_0=\{u_1,\dots ,u_r\}$, $V_0=\{v_1,\dots ,v_{r'}\}$, be the center and the set of leaves of $K_{1,r}$ and $K_{1,r'}$, respectively. Let $X=(U_0\times \{v_0\})\cup (\cup_{j=1}^{r'-1} \{(u_r,v_j)\}) \cup \{(u_{r-1}, v_{r'})\}$ and $Y=(\{u_0\}\times V_0)\cup (\cup_{i=1}^{r-1} \{(u_i,v_1)\}) \cup \{(u_r, v_{r'})\}$. Note that, for each vertex $x\in X \subseteq \overline{Y}$, there exists exactly one vertex $y \in Y$ such that $N(y) \cap X=\{x\}$. So, $V(K_{1,r}\Box K_{1,r'}) \setminus X$ is a super dominating set of $K_{1,r} \square K_{1,r'}$ with cardinality $(r+1)(r'+1)-(r+r')=rr'+1$, and thus $\gamma_{\rm sp}(K_{1,r}\Box K_{1,r'})\le rr'+1$. Therefore, $\gamma_{\rm sp}(K_{1,r}\Box K_{1,r'}) = rr'+1$.
%Let $X=(\{u_0\}\times V_0)\cup (\cup_{i=2}^r \{(u_i,v_1)\})$ and $Y=(\{u_1\}\times V_0)\cup (\cup_{i=2}^r \{(u_i,v_0)\})$. Notice that every vertex in $X$ has exactly one neighbour in $Y$, and vice versa. Hence, $V(K_{1,r}\Box K_{1,r'}) \setminus Y$ is a super dominating set and so
%$\gamma_{\rm sp}(K_{1,r}\Box K_{1,r'})\le (r+1)(r'+1)-|Y|=rr'+2.$
\end{proof}

\section*{Acknowledgements}
This research was supported in part by  the Spanish  government  under  the grants MTM2016-78227-C2-1-P and PRX17/00102.

\bibliographystyle{elsart-num-sort}
%\bibliography{D:/Dropbox/References/MiBilio}

%\bibliography{MiBilio}
%\bibliography{C:/Users/47275715-c/Dropbox/References/MiBilio}

\end{document}